\documentclass[11pt]{article}
\usepackage{amsthm}
\usepackage{amsfonts}
\usepackage[T1]{fontenc}
\usepackage{epsfig}

\usepackage[lflt]{floatflt}
\usepackage{epsfig}
\usepackage{graphicx}
\usepackage{amsmath}
\usepackage{amssymb}
\usepackage{color}
\usepackage{subfigure}
\usepackage[english]{babel}
\usepackage{empheq}
\usepackage{caption}
\usepackage{stmaryrd}

\usepackage{makeidx}
\usepackage{latexsym}
\usepackage{enumerate}
\newtheorem{theo}{Theorem}
\newtheorem{defi}{Definition}
\newtheorem{prop}{Proposition}

\newtheorem{coro}{Corollary}
\newtheorem{remark}{Remark}

\newcommand{\bbN}{{\mathbb N}}
\newcommand{\bbR}{{\mathbb R}}

\newcommand{\al}{\alpha}

\newcommand{\p}{\partial}
\newcommand{\be}{\beta}
\newcommand{\G}{\Gamma}

\newcommand{\dd}{{\rm d}}

\usepackage{anysize}
\marginsize{2.5cm}{2.5cm}{2.5cm}{3cm}
\begin{document}
\begin{center}\emph{}
\LARGE
\textbf{ The similarity method and explicit solutions for the fractional space one-phase Stefan problems}
\end{center}

                   \begin{center}
                  {\sc Sabrina D. Roscani,  Domingo A. Tarzia and Lucas Venturato\\
 CONICET - Depto. Matem\'atica, FCE, Universidad Austral, Paraguay 1950, S2000FZF Rosario, Argentina \\
(sroscani@austral.edu.ar, dtarzia@austral,edu.ar, lventurato@austral.edu.ar)}
                   \vspace{0.2cm}

       \end{center}
      
\small

\noindent \textbf{Abstract:} 
In this paper we obtain self-similarity solutions for a one-phase one-dimensional  fractional space one-phase Stefan problem in terms of the three parametric Mittag-Leffler function $E_{\al,m,l}(z)$. We consider Dirichlet and Newmann conditions at the fixed face, involving  Caputo fractional space derivatives of order $0<\al<1$. We recover the solution for the classical  one-phase Stefan problem  when the order of the Caputo derivatives approaches one. 
  


\section{Introduction}\label{Sec:Introduction}

This paper deals with a fractional space Stefan problem. More precisely, we consider a  phase-change problem where the heat flux is  modeled through fractional integrals, and the governed equation is a fractional diffusion equation.

Fractional diffusion equations are a wide scope which could be related to different theories, all of them, converging to  the classical diffusion equation which, in a simple one dimensional form can be written as 
\begin{equation}\label{ClDifEq}
u_t(x,t)=u_{xx}(x,t)+ F(x,t), \qquad x\, \in \, \Omega\subset \bbR, \, \, t \, \in \, (0,T).
\end{equation}  
Regarding fractional diffusion equations for Caputo and Riemann-Liouvlle derivatives and its applications, a complete view of the state-of-the-art can be found in \cite{Povstenko}. A rigorous  mathematical analysis is presented in  \cite{LuMaPa-TheFundamentalSolution, Pskhu:2009} and for applications we refer the reader to \cite{Hilfer, FM-libro}.

We will work with the following fractional diffusion equation  where a Caputo derivative on the spatial variable is involved.   

\begin{equation}\label{FDE-q-Cap}
u_t(x,t)=\frac{\p}{\p x}\,^C_0D_x^{\al} u(x,t), \quad x \in \Omega\times(0,T), \,\, \al\, \in \, (0,1). 
\end{equation}
Recall that $^C_0D_x^{\al}$ is the fractional Caputo derivative of order $\al$ respect on the spatial variable given by 

\begin{equation}\label{Caputo}
^C_0D_x^{\al} u(x,t)=\, _0I^{1-\al}[u_x(\cdot,t)](x,t)=\frac{1}{\G(1-\al)}\int_0^x \frac{u_x(p,t)}{(x-p)^\al}\dd p
\end{equation}
and $_0I^{1-\al}$ is the fractional integral of Riemann-Liouville in the spatial variable of order $1-\al$, defined for every summable function $f$ as 
\begin{equation}\label{IntRL} _0I^{\beta}f(x)=\frac{1}{\G(\beta)}\int_0^x f(p)(x-p)^{\beta-1} \dd p,\quad 0<\beta <1.\end{equation}
Henceforth, the parameter $\al$ related to the fractional order will be a real number between 0 and 1 and the subscript $x$ in  fractional integral and derivatives will be omitted in the context of one variable functions as in the previous equality.

It is worth noting that equation \eqref{FDE-q-Cap} is a consolidated model to anomalous diffusion \cite{delCasNe:2006, MeKl:2000, Povstenko} whereas it was proved in 
\cite{BaMe:2018} that the equation 
$$ u_t(x,t)=\,^C_0D_x^{\al+1} u(x,t)   $$
cannot provide a suitable model  for anomalous diffusion. 

The fractional Stefan problem for the one-dimensional time-fractional diffusion equation was recently studied. Different  models are presented in \cite{GeVoMiDa:2013}, \cite{RoBoTa:2018} and \cite{VoFaGa:2013}. A rigorous existence analysis of self-similar solutions was done in \cite{KuRy:2020}, and results related to explicit solutions were established in \cite{RoSa:2013, RoTa:2017-TwoDifferent, RoCaTa:2020}   and references therein. 

Space-fractional Stefan problems were proposed in \cite{Vo:2014} and the literature about this topic is currently emerging.  
Recently, K. Ryszewska provides  in \cite{Rys:2020} the mathematical analysis of a one dimensional, one-phase free boundary problem governed by a space-fractional diffusion equation. In that article, it is proved that  the problem to find a pair $\left\{u,s\right\}$ verifying that 

\begin{equation}\label{Rys-Pb}
\begin{array}{ll}
u_t(x,t)=\frac{\p}{\p x}\, ^C_0D_x^{\al} u(x,t) & 0<x<s(t), 0<t<T,\\
u_x(0,t)=0 &  0<t<T,\\
u(s(t),t)=0 & 0<t<T,\\
u(x,0)=u_0(x) & 0<x<s(0)=b, \\
\dot{s}(t)=-\lim\limits_{x\rightarrow s(t)^-}\,_0^C D_x^{\al} u_x(x,t) & 0<t<T
\end{array}
\end{equation}
 has a unique solution under suitable regularity on the initial condition and the assumption that $b$ is a positive number.\\

In this paper two  similar problems are treated. Let the parabolic domain
$$Q_{s,T}=\left\{ (x,t) \colon \, 0< x< s(t), \, 0<t < T \right\}$$
be. We consider two instantaneous melting fractional space  Stefan problems. The first one addressed  with a Dirichlet condition:   \textit{Find the pair of functions} $u\colon Q_{s,T}\rightarrow \bbR$ and $s\colon [0,T]\rightarrow \bbR$ with sufficiently regularity such that 
\begin{equation}\label{FSSP-al-D}
\begin{array}{ll}
{u}_t(x,t)=\frac{\p}{\p x} ^C_0D_x^{\al} u(x,t) & 0<x<s(t), 0<t<T,\\
u(0,t)=g(t)  &  0<t<T,\\
u(s(t),t)=U_m & 0<t<T,\\
s(0)=0, & \\ 
\dot{s}(t)=-\lim\limits_{x\rightarrow s(t)^-} \,^C_0D_x^{\al} u(x,t) & 0<t<T.
\end{array}
\end{equation}

And the second one addressed with a Neumann condition:
  \textit{Find the pair of functions} $w\colon Q_{s,T}\rightarrow \bbR$ and $r\colon [0,T]\rightarrow \bbR$ with sufficiently regularity such that 
\begin{equation}\label{FSSP-al-N}
\begin{array}{ll}
w_t(x,t)=\frac{\p}{\p x} ^C_0D_x^{\al} w(x,t) & 0<x<r(t), 0<t<T,\\
^C_0D_x^{\al} w(0^+,t)=-h(t)  &  0<t<T,\\
w(s(t),t)=U_m & 0<t<T,\\
r(0)=0, & \\ 
\dot{r}(t)=-\lim\limits_{x\rightarrow r(t)^-} \,^C_0D_x^{\al} w(x,t) & 0<t<T.
\end{array}
\end{equation}

Note that a Neumann condition  \eqref{FSSP-al-N}$-(ii)$ is different than \eqref{Rys-Pb}$-(ii)$ and it will be justified in the next section, where the model is presented.

The structure of the paper is the following: We  derived problems \eqref{FSSP-al-D} and \eqref{FSSP-al-N}   from physical assumptions in Section 2. Then, some useful properties related to the special functions involved in the self-similarity solutions are presented in Section 3. In Section 4, we apply the similarity method in order to obtain a solution as a function of the three-parameter Mittag-Leffler function $E_{\al,m,l}(z)$ and the non-negative property of the function defined by \eqref{defsigma}. In Section 5 and in Section 6 we obtain the unique explicit solution for the fractional space one-phase Stefan problem with a Dirichlet \eqref{FSSP-al-D} and a fractional heat flux condition \eqref{FSSP-al-N} at the fixed face $x=0$ respectively.

\section{The mathematical model for instantaneous phase change}

\noindent Consider an instantaneous phase change problem corresponding to the melting of a  semi--infinite slab  ($0\leq x<\infty$) of  a material, which is initially at the melting temperature  $U_m$, by imposing a temperature or a heat flux condition at the fixed face $x=0$. All the thermophysical parameters are considered to be constants.  \\

The notation related to heat conduction with its corresponding physical dimensions are given in the next table:

\begin{equation}\label{medidas}
\begin{array}{ccc}
\left[u\right] & \textsl{temperature} & \textbf{T}\\
\left[k\right] & \textsl{thermal conductivity} & \frac{{\textbf{m X}}}{{\textbf{Tt}^3}} \\
\left[\rho\right] & \textsl{mass density} &  \frac{{\textbf {m }}}{{\textbf {X}^3}}  \\
\left[c\right] & \textsl{specific heat} &   \frac{\textbf{X}^2 }{\textbf{T}\textbf{t}^2}\\
\left[d\right]=\left[\frac{k}{\rho c}\right] & \textsl{ diffusion coefficient} &  \frac{\textbf{X}^2}{\textbf{t}}\\
\left[l \right]& \textsl{latent heat per  unit  mass} & \frac{\textbf{X}^2 }{\textbf{t}^2}
\end{array}
\end{equation}
 where $\textbf{T}: temperature$, $\textbf{t}: time$, $\textbf{m}:mass$, $\textbf{X}: position$.

Let $u=u(x,t)$ be the temperature and  let $q=q(x,t)$ be the heat flux of the material at position $x$ and time $t$. Let $x=s(t)$ be the function representing the (unknown) position of the free boundary (phase change interface) at time $t$ such that $s(0)=0$. 

Suppose that, at every time $t$   the heat flux at a position $x$ is a generalized weighted sum of the classical fluxes occurring at every position from the initial position to the current one, where the nearest local fluxes are more relevant than the farthest.
That is, we model the heat flux in the slab by the expression
\begin{equation}\label{q-elegido-0}
 q(x,t)=-\nu_\al \frac{1}{\G(1-\al)} \int_0^x k\,u_x(p,t)(x-p)^{-\al} \dd p=  -\nu_\al k\,_0I^{1-\al}_x\frac{\p  u}{\p x}(x,t).
\end{equation}

Equation \eqref{q-elegido-0} can be expressed in terms of Caputo derivatives as follows
\begin{equation}\label{q-elegido}
 q(x,t)=-\nu_\al k \, ^C_0D_x^{\al} u(x,t).
\end{equation}

Note that $k$ is the thermal conductivity whereas the parameter  $\nu_\al$ has been added to preserve the consistency with respect to the units of measure in equation \eqref{q-elegido-0} such that 
\begin{equation}\label{lim-nu-al}
\displaystyle\lim_{\al\nearrow 1}\nu_{\al}=1.
\end{equation}
 
 From the units of measure given in (\ref{medidas}), we have:
 \begin{equation}\label{med-J}
  \left[ q  \right]=\frac{\textbf{m}}{\textbf{t}^3},
 \end{equation}
 then
 \begin{equation}\label{med-IalJ}
 \left[\nu_\al\right] \left[ I^{1-\al}_x k u_x(x,t) \right]=\left[\nu_\al\right] \left[\frac{1}{\G(1-\al)} \int_{0}^x\frac{k u_x(p,t)}{(x-p)^\al}\dd p \right]= \left[\nu_\al\right]
\frac{\textbf{m}}{\textbf{t}^3}\textbf{m}^{1-\al}. 
 \end{equation}
 Therefore $\left[\nu_\al\right]= \textbf{m}^{\al-1}$.

Now, let us derive the two governing equations of the problem.  From the first principle of the thermodynamics, we have that
\begin{equation}\label{cont eq}
\rho c \frac{\p u}{\p t}(x,t)=-\frac{\p q}{\p x}(x,t).
\end{equation}

Then, by replacing \eqref{q-elegido} in the continuity equation \eqref{cont eq}, the governing equation (now with all the physical parameters) becomes

\begin{equation}
\rho c \frac{\p u}{\p t}(x,t)=\nu_\al k \frac{\p }{\p x}D_x^{\al}u(x,t)
\end{equation}
which in terms of the fractional diffusivity constant, defined by 
$$\lambda_\al=\nu_\al \lambda, \qquad \lambda=\frac{k}{\rho c},$$
is expressed as

\begin{equation}\label{gov eq}
\frac{\p u}{\p t}(x,t)=\lambda_\al  \frac{\p }{\p x}D_x^{\al}u(x,t).
\end{equation}

Respect on the interface we are considering a sharp model  where the solid phase is at constant temperature equal to $U_m$.   Then the  Rankine--Hugonoit conditions gives the condition
\begin{equation}\label{RK-eq}
\llbracket \textbf{q}  \rrbracket^s_l   =-\rho l \dot{s}(t).
\end{equation}
where  the double brackets represents the difference between the limits of the fluxes  from the solid phase and the liquid phase and $l$ is the latent heat of fusion by unit of mass. The fractional Stefan condition then, is obtained from  \eqref{q-elegido} and \eqref{RK-eq}  and it is given by

\begin{equation}\label{Fr-St-Cond}
\rho l \dot{s}(t)=- \nu_\al k \lim\limits_{x\rightarrow s(t)^-}(D_x^\alpha u)(x,t), \quad    t\in (0,T),
\end{equation}
which, for simplicity, will be written as 

\begin{equation}\label{Fr-St-Cond}
\rho l \dot{s}(t)=- \nu_\al k D_x^\alpha u(s(t),t), \quad    t\in (0,T).
\end{equation}

Then, by supposing that the melting temperature is given by $ u(s(t),t)=U_m$, we can address the problem with Dirichlet type conditions
\begin{equation}\label{Dir-cond}
u(0,t)=g(t),
\end{equation}
or by considering a Neumann boundary condition at $x=0$ which according to \eqref{q-elegido} it is given by
\begin{equation}\label{Sp-Fr-Flux2}
\lim\limits_{x\rightarrow 0^+}D_x^\al u(x,t)=-h(t)
\end{equation}
 where $g(t)\geq U_m$ for every $t$ and $h$ is a positive function according to the melting  model considered.

Thus, the one-dimensional fractional space one-phase free-boundary problems for Dirichlet and Neumann conditions at $x=0$ are given respectively by the following expressions:

\begin{equation}\label{FSP-D}
 \begin{array}{lll}
(i) & \frac{\partial}{\partial t}u(x,t)-\lambda_\al\frac{\partial}{\partial x}\,^C_0D_x^\alpha u (x,t)=0, & 0<x<s(t),0<t<T,\\
(ii) & u(0,t)=g(t), &  0<t<T,\\
(iii) & u(s(t),t)=U_m, &  0<t<T,\\
(iv) & s(0)=0, &  \\
(v) & \rho l \dot{s}(t)=- \nu_\al k (\,^C_0D_x^\alpha u)(s(t),t), &  0<t<T,
\end{array}
\end{equation}
and 
\begin{equation}\label{FSP-N}
 \begin{array}{lll}
(i) & \frac{\partial}{\partial t}w(x,t)-\lambda_\al \frac{\partial}{\partial x}\,^C_0D_x^\alpha w (x,t)=0, & 0<x<s(t),0<t<T,\\
(ii) & \lim\limits_{x\rightarrow 0^+} \,^C_0D^\al_x w(x,t)=-h(t), &  0<t<T,\\
(iii) & w(s(t),t)=U_m, & 0<t<T,\\
(iv) & s(0)=0, & \\
(v) & \rho l \dot{s}(t)=-  \nu_\al k(\,^C_0D_x^\alpha w)(s(t),t), &   0<t<T.
\end{array}
\end{equation}

In \cite{Vo:2014} the quasi-stationary case was solved. There, it was shown that the pair 
\begin{equation}\label{Fr-StationarySol}
u(x,t)=1-\frac{x^\al}{\left[\G(2+\al)\right]^{\frac{\al}{1+\al}}t^{\frac{\al}{1+\al}}}, \qquad s(t)=\left[\G(2+\al)\right]^{\frac{1}{1+\al}}t^{\frac{1}{1+\al}}
\end{equation} 
is a solution to problem

\begin{equation}
 \begin{array}{lll}
(i) & \frac{\partial}{\partial x}\,^C_0D_x^\alpha u (x,t)=0, & 0<x<s(t),0<t<T,\\
(ii) & u(0,t)=1, &  0<t<T,\\
(iii) & u(s(t),t)=0, &  0<t<T,\\
(iv) & \dot{s}(t)=-(\,^C_0D_x^\alpha u)(s(t),t), &   0<t<T.
\end{array}
\end{equation}

\section{Some basics of the fractional calculus involved in this model}

The definitions of fractional integral of Riemann-Liouville and Caputo derivative were given in \eqref{IntRL} and \eqref{Caputo} respectively. Recall that the Riemann-Liouville derivative of order $\al$ is defined for every absolute continuous function $f$ as 

\begin{equation}\label{RieLiou}
^{RL}_0D_x^{\al} f(x)= \frac{d}{dx}\, _0I^{1-\al}f(x)=\frac{1}{\G(1-\al)}\frac{d}{dx}\int_0^x f(p)(x-p)^{-\al} \dd p.
\end{equation}

\begin{prop}\label{propo frac}\cite{Diethelm} The following properties involving the fractional integrals and derivatives of order $\al \in (0,1)$ hold:
\begin{enumerate}
\item \label{RL inv a izq de I} The  \textsl{fractional Riemann--Liouville derivative} is a left inverse operator of the \textsl{fractional Riemann--Liouville integral} of the same order  $\al\in \bbR^+$. If $f \in L^1(a,b)$, then
$$^{RL}_{a}D^{\al}\,_{a}I^{\al}f(x)=f(x)  \quad a.e. \text{in } \, (a,b).$$

\item The fractional  Riemann--Liouville integral, in general,  is not a left inverse operator of the fractional derivative of Riemann--Liouville.\\
 
In particular, we have 
$ _{a}I^{\al}(^{RL}_{a}D^{\al}f)(x)=f(x) - \dfrac{_{a}I^{1-\al}f(a^+)}{\G(\al)(x-a)^{1-\al}} $ for every  $x \in [a,b].$ 

\item\label{caso part I inv de RL} If there exist some function $\phi \, \in L^1(a,b)$ such that $f=\,_aI^\al \phi$, then 
$$_{a}I^{\al}\,^{RL}_{a}D^{\al} f(x)=f(x)  \quad \forall \, x\in [a,b].$$

\item\label{relacion RL-C} If 
 $f\in AC[a,b],$ then  
$$ ^{RL}_{a}D^{\al}f (x)=\frac{f(a)}{\G(1-\al)}(x-a)^{-\al}+\, ^C_{a} D^{\alpha}f(x) \quad a.e. \text{in } \, (a,b).$$ 

\item\label{propClave} For every  $f \in AC[a,b]$ such that $_aI^{1-\al}f' \in AC[a,b]$ it holds that 

$$\frac{d}{d x}\,^C_aD^\alpha f(x)=\,_a^{RL}D^\alpha(f')(x), \qquad a.e. \text{ in } (a,b).$$
\end{enumerate}
\end{prop}

\begin{prop} \cite{Samko} The following limits hold:
\begin{enumerate}
\item If we set $_aI^0=Id$, the identity operator, then for every $ f \, \in L^1(a,b)$,   
$$ \displaystyle\lim_{\al\searrow 0}\, _aI^\al f(x) =\,_aI^0f(x)=f(x) \qquad a.e. \text{ in } (a,b). $$

\item  For every  $f \in AC[a,b]$ the limits hold $a. e.$ in $(a,b)$, 
$$ \displaystyle\lim_{\al\nearrow 1}\,_a^{RL}D^\al f(x) = f'(x),\hspace{1cm}  \displaystyle\lim_{\al\searrow 1}\,_a^{RL}D^\al f(x) = f'(x) \hspace{1cm} \text{and} \hspace{1cm} \displaystyle\lim_{\al\nearrow 1}\,_a^CD^\al f(x) = f'(x). $$
If additionally there exists $f'(0^+)$, it holds that 
 $$ \displaystyle\lim_{\al\searrow 1}\,_a^CD^\al f(x) = f'(x)-f'(0^+).$$

\end{enumerate}

\end{prop}

There are not many functions such that we can make a direct computation of its fractional integral or derivative. From simple calculations we know that the fractional integral and derivative of powers are given by
\begin{equation}\label{Int-de-pot}
_aI^{\al}\left((x-a)^{\be}\right)=\frac{\G(\be + 1)}{\G(\be+\al+1)}(x-a)^{\be+\al }, \qquad \text{ for every } \quad  \be>-1, 
\end{equation}
and 
\begin{equation}\label{DRL-de-pot}
_a^{RL}D^{\al}\left((x-a)^{\be}\right)=\begin{cases}\frac{\G(\be + 1)}{\G(\be-\al+1)}(x-a)^{\be-\al} & \text{ if } \be \neq \al - 1,\\ 
0 & \text{ if } \be=\al-1.
\end{cases} 
\end{equation}

Besides, in  \cite{KiSi:1996}  some computations of  integrals and derivatives  of some special cases related to a three-parametric Mittag--Leffler function were proved. They are in the next proposition (see \cite{KiSi:1996}: Theorem 2 and Theorem 4).

\begin{defi} Let $\al>0, m>0$, and $l$ such that $\al(jm+l)\neq -1,-2,-3,\dots \,\, (j=0,1,2,\dots)$. The three-parametric Mittag-Leffler function $E_{\al,m,l}(z)$ is defined by
\begin{equation}\label{ML-3param}
E_{\al,m,l}(z)=\sum_{n=0}^\infty c_n z^n,\,\,\text{ with }\,\,c_0=1,\,\, c_n=\prod\limits_{j=0}^{n-1}\frac{\Gamma(\al(jm+l)+1)}{\Gamma(\al(jm+l+1)+1)},\,\, (n=1,2,3,\dots).
\end{equation}

\end{defi}

\begin{remark}\label{ML3particulares}
In particular, $E_{1,1,0}(z)=e^z$ and we recover the classical Mittag-Leffler function for $m=1$ and $l=0$ $E_{\al,1,0}(z)=E_{\al}(z)$.   
Also, a two parametric Mittag--Leffler function is recovered for the case $E_{\al,1,l}(z)=\Gamma(\al l+1)E_{\al,\al l+1}(z)$. And the special case of our interest which is 
  $E_{1,2,1}\left(-\frac{z^2}{2}\right)=e^{-\left(\frac{z}{2}\right)^2}.$
\end{remark}

\begin{prop}\label{MLintder} Let $\al>0, m>0$ and $a\neq 0$ be. 
\begin{enumerate}
\item If $l>-\frac{1}{\al}$, then for every $x\in \bbR^+$ it holds that
$$(_{0}I^\al[p^{\al l}E_{\al,m,l}(a p^{\al m})])(x)=\frac{1}{a}x^{\al(l-m+1)}[E_{\al,m,l}(a x^{\al m})-1],$$
\item If $l>m-1-\frac{1}{\al}$, then for every $x\in \bbR^+$ it holds that
$$(_{0}^{RL}D^\al[p^{\al (l-m+1)}E_{\al,m,l}(a p^{\al m})])(x)=\frac{\Gamma(\al(l-m+1)+1)}{\Gamma(\al(l-m)+1)}x^{\al(l-m)}+ax^{\al l}E_{\al,m,l}(a x^{\al m}).$$
If further $\al (l-m)=-j$ for some $j=1,2,\dots,-[-\al]$, then
$$(_{0}^{RL}D^\al[p^{\al (l-m+1)}E_{\al,m,l}(a p^{\al m})])(x)=ax^{\al l}E_{\al,m,l}(a x^{\al m}).$$
\end{enumerate}
\end{prop}

We will focus on the function 
$\sigma_\al(z)=z^{\alpha-1}E_{\alpha,1+\frac{1}{\alpha},1}\left(-\frac{z^{1+\alpha}}{1+\alpha}\right)$ defined in $\bbR^+$, which will take part in the explicit solutions that will be presented in the next section.

By applying Proposition \ref{MLintder} to function $\sigma_\al$ for the particular case:
$$l=1,\quad,m=1+\frac{1}{\al},\quad a=-\frac{1}{1+\al},$$
it yields that 
\begin{equation}\label{sigmaint}
\left(_{0}^{RL}D_x^\al\left[p^{\al-1}E_{\al,1+\frac{1}{\alpha},1}\left(-\frac{p^{1+\al}}{1+\al}\right)\right]\right)(x)=-\frac{1}{1+\al}x^{\al}E_{\al,1+\frac{1}{\alpha},1}\left(-\frac{x^{1+\al}}{1+\al}\right)
\end{equation}
for every $x\in \bbR^+$.

Moreover, the next interesting convergence holds.  
\begin{prop}\label{prop5}
Let $\displaystyle f_\al (x):=\int_0^x w^{\alpha-1}E_{\alpha,1+\frac{1}{\alpha},1}\left(-\frac{w^{1+\alpha}}{1+\alpha}\right)\dd w$, for $\al\in(0,1)$ be. Then, we have
$$\lim\limits_{\al\nearrow 1}f_\al (x)=\sqrt{\pi}f\left(\frac{x}{2}\right), \quad \text{for every }\, x\, \in \, \bbR^+_0.$$
 where $f(x):={\rm erf}(x)$  is the error function defined in $\bbR^+_0$ by the expression ${\rm erf}(x):=\frac{2}{\sqrt{\pi}}\displaystyle\int_0^x e^{-s^2}\dd s.$

\end{prop}

\begin{proof}
Note that 
$$ \sigma_\al(x)=x^{\alpha-1}E_{\alpha,1+\frac{1}{\alpha},1}\left(-\frac{x^{1+\alpha}}{1+\alpha}\right)= \sum\limits_{n=0}^\infty c_n (-1)^n \frac{x^{(n+1)(1+\alpha)-2}}{(1+\alpha)^n}, $$
where the series is convergent in $\bbR$, and then it converges uniformly in every compact subset of $\bbR$. Hence, integrating term by term in the series, the following expression to $f_\al$ holds
\begin{equation}\label{f_al-Serie} f_\al (x)=x^\alpha\sum\limits_{n=0}^\infty\frac{ c_n }{[(n+1)(1+\alpha)-1]}\left(-\frac{x^{1+\alpha}}{1+\alpha}\right)^n.\end{equation}
From \eqref{f_al-Serie} we deduce that $f_\al$ is an analytic function for all $x>0$ and the limit when $\al\rightarrow 1^-$ can be computed term by term. Then, taking into account that 
\begin{equation}
\begin{split}
\lim\limits_{\alpha\nearrow 1}c_n&=\lim\limits_{\alpha\nearrow 1}\prod\limits_{j=0}^{n-1}\frac{\Gamma((j+1)(1+\alpha))}{\Gamma((j+1)(1+\alpha)+\alpha)}=\frac{1}{2^n n!}
\end{split}
\end{equation}
we have
$$
\lim\limits_{\alpha\nearrow 1} f_\alpha(x)=x\sum\limits_{n=0}^\infty\frac{1}{2^n n![2(n+1)-1]}\left(-\frac{x^2}{2}\right)^n
=\sum\limits_{n=0}^\infty \frac{(-1)^n }{2^{2n} n!(2n+1)}x^{2n+1}
=\sqrt{\pi}\,{\rm erf}\left(\frac{x}{2}\right).
$$
\end{proof}

\section{The self-similar solution in terms of the Mittag-Leffler function and its properties}

The aim of this section is to obtain an exact solution to problems  \eqref{FSP-D} and \eqref{FSP-N}. For simplicity, all the thermophysical parameters will be considered as constants equals to one. 

First, we will look for a self-similar solution through the method of similarity variables \cite{Cannon, Pincho-Rubi, Tarzia}. Suppose that $u=u(x,t)$ is a solution to the space fractional diffusion equation \eqref{FDE-q-Cap} and let the function $u_\lambda$ be defined by
\begin{equation}\label{cambiovariables}
u_\lambda(x,t)=u\left(\frac{x}{\lambda},\frac{t}{\lambda^b}\right),
\end{equation}
for $b\in \bbR$ and $\lambda>0$.

\begin{prop}
A function $u=u(x,t)$ is a solution to \eqref{FDE-q-Cap} in $\Omega\times(0,T)$  if and only if $u_\lambda= u_\lambda(x,t)$ is a solution to \eqref{FDE-q-Cap} in $\widetilde{\Omega}\times(0,\widetilde{T})$ for $b=1+\al$,  for all $\lambda>0$, where $\widetilde{\Omega}=\frac{1}{\lambda}\Omega$, $(0,\widetilde{T})=\left(0,\frac{T}{\lambda^{1+\al}}\right)$.
\end{prop}

\begin{proof}
Let us define  the function  $ u_\lambda(x,t)\colon= u(\bar{x},\bar{t}) $ where $x=\lambda \bar{x}$ and $t=\lambda^b \bar{t}$. It is straightforward to see that 
\begin{equation}\label{met-sem-1}
\frac{\p}{\p t}u_\lambda(x,t)=u_{\bar{t}}(\bar{x},\bar{t})\frac{1}{\lambda^b}.
\end{equation} 
\begin{equation}\label{met-sem-2}
\frac{\p}{\p x}u_\lambda(x,t)=u_{\bar{x}}(\bar{x},\bar{t})\frac{1}{\lambda}.
\end{equation} 
\begin{equation}\label{met-sem-3}
^C_0D_x^{\al} u_\lambda (x,t)=\lambda^{-\al}\, ^C_0D_{\bar{x}}^{\al} u(\bar{x},\bar{t}).
\end{equation} 
and 
\begin{equation}\label{met-sem-4}
\frac{\p}{\p x}\,^C_0D_x^{\al} u_\lambda (x,t)=\lambda^{-\al-1}\frac{\p}{\p \bar{x}}\, ^C_0D_{\bar{x}}^{\al} u(\bar{x},\bar{t}).
\end{equation} 
Then, from \eqref{met-sem-1}, \eqref{met-sem-2} and \eqref{met-sem-4} it follows that 
 \begin{equation}\label{met-sem-5}
\frac{\p}{\p t}u_\lambda(x,t)-\frac{\p}{\p x}\,^C_0D_x^{\al} u_\lambda (x,t)= \lambda^{-b}u_{\bar{t}}(\bar{x},\bar{t})-  \lambda^{-\al-1}\frac{\p}{\p \bar{x}}\,^C_0D_{\bar{x}}^{\al} u(\bar{x},\bar{t}).
\end{equation} 
From equality \eqref{met-sem-5} the thesis holds.
\end{proof}

The scaling in the previous result indicates that the ratio $\frac{x}{t^{\frac{1}{1+\alpha}}}$ plays an important role in equation \eqref{FDE-q-Cap}. This fact  suggests us to search for a solution
$u(x,t)=\theta \left(\frac{x}{t^{\frac{1}{1+\alpha}}}\right)$. Thus we define the one variable function  
\begin{equation}\label{Theta}
\theta(z):= u(x,t),
\end{equation}
where $z$ is the similarity variable defined as
\begin{equation}\label{SimVar}
z:= \frac{x}{t^{\frac{1}{1+\alpha}}}.
\end{equation}
Now, we  apply the chain rule in order to obtain an ordinary fractional differential equation for the function $\theta=\theta(z)$.
The next calculation follows from the chain rule.
\begin{equation}\label{eq6}
\begin{split}
\frac{\partial}{\partial t}u (x,t)=-\frac{z}{(1+\alpha)t}\theta'(z).
\end{split}
\end{equation}
Also, by making the substitution $w=\frac{p}{t^{\frac{1}{1+\alpha}}}$, it follows that
\begin{equation}\label{eq7}
\begin{split}
\frac{\partial}{\partial x}\left(^C_0D_{x}^{\al} u (x,t)\right)&=\frac{\partial}{\partial x}\left(\frac{1}{\Gamma(1-\alpha)}\int_0^z \frac{\theta'(w)}{t^{\frac{\alpha}{1+\alpha}}(z-w)^\alpha} \dd w\right)\\
&=\frac{1}{t}\frac{\partial}{\partial z}\,^C_0D_{z}^{\al} \theta(z).
\end{split}
\end{equation}

From \eqref{eq6} and \eqref{eq7}, we deduce

\begin{equation}\label{eq9}
\begin{split}
0=\frac{\partial}{\partial t}u (x,t) -\frac{\partial}{\partial x}\,^C_0D_{x}^{\al}  u (x,t)=-\frac{1}{t}\left[\frac{z}{1+\alpha}\theta'(z)+\frac{\partial}{\partial z}\,^C_0D_{z}^{\al} \theta(z)\right],
\end{split}
\end{equation}
and then,
\begin{equation}\label{eqdiftheta}
\frac{z}{1+\alpha}\theta'(z)+\frac{\partial}{\partial z}\,^C_0D_{z}^{\al} \theta(z)=0.
\end{equation}

Reciprocally, if $\theta$ is a solution to \eqref{eqdiftheta}, we can go back over previous calculations and obtain that  $u$ is  a solution of  \eqref{FDE-q-Cap}. More precisely:

\begin{prop}The function $u$ is a solution to equation \eqref{FDE-q-Cap} if and only if  function $\theta$ defined by \eqref{Theta} with the similarity variable \eqref{SimVar} is a solution to equation \eqref{eqdiftheta}.
\end{prop}

Now, we seek for a solution to \eqref{eqdiftheta}. Making the substitution $\sigma(z)=\theta'(z)$, and using Proposition \ref{propo frac} (part \ref{propClave}), we convert \eqref{eqdiftheta} into the next equation
\begin{equation}\label{eq9.5}
\begin{split}
\frac{z}{1+\alpha}\sigma(z)+{^R}{^L}D_z^\alpha(\sigma)(z)=0.
\end{split}
\end{equation}
From \cite{KiSi:1996} we know that a solution to \eqref{eq9.5} is given by
\begin{equation}\label{eq10}
\begin{split}
\sigma(z)&=z^{\alpha-1}E_{\alpha,1+\frac{1}{\alpha},1}\left(-\frac{z^{1+\alpha}}{1+\alpha}\right)=\sum\limits_{n=0}^\infty c_n (-1)^n \frac{z^{(n+1)(1+\alpha)-2}}{(1+\alpha)^n},
\end{split}
\end{equation}
where $c_n$ is given in \eqref{ML-3param}.

Hence,
\begin{equation}
\begin{split}
\theta(z)&=A+B\int_0^z w^{\alpha-1}E_{\alpha,1+\frac{1}{\alpha},1}\left(-\frac{w^{1+\alpha}}{1+\alpha}\right)\dd w=A+B\sum\limits_{n=0}^\infty\frac{ c_n (-1)^n }{(1+\alpha)^n}\frac{z^{(n+1)(1+\alpha)-1}}{(n+1)(1+\alpha)-1}\\
\end{split}
\end{equation}
is a solution to equation \eqref{eqdiftheta} for arbitrary real constants $A$ and $B$.

\begin{remark} Note that the unique continuous solution to \eqref{eq9.5} at $z=0$ is the null function, that is, the solution such that $\theta'(0^+)=0$. But, adressing the problem with initial conditions in terms of fractional integrals, we obtain solutions with a singularity at $z=0$ that verify the requirest initial condition.
\end{remark}

\begin{remark}
We can say that $\theta$ is an absolutely continuous function, since $\displaystyle\theta(z)=\theta(0)+\int_0^z \theta'(w)\dd w$.  Therefore, by Proposition \ref{propo frac}, $\frac{\partial}{\partial z}\,^C_0D_z^\alpha\theta(z)={^R}{^L}D_z^\alpha(\theta')(z)$.\end{remark}

Hereinafter we denote by
\begin{equation}\label{defsigma}
\sigma_\al(w):= w^{\alpha-1}E_{\alpha,1+\frac{1}{\alpha},1}\left(-\frac{w^{1+\alpha}}{1+\alpha}\right).
\end{equation}

\begin{prop}\label{solgeneral}
For every $A, B \in \bbR$, the function $u\colon \bbR^+_0\times(0,T)\rightarrow \bbR$ such that 
\begin{equation}\label{solu}
u(x,t)= A+B \int_0^{x/t^{\frac{1}{1+\al}}}\sigma_\al(w)\dd w.
\end{equation}
is a solution to the space-fractional diffusion equation \eqref{FDE-q-Cap}.
\end{prop}
\proof The  proof is a direct consequence from the chain rule, property \eqref{propClave} of Proposition \ref{propo frac} and expression \eqref{sigmaint}.
\endproof

\begin{remark} It is also interesting the series approach in the aim to prove that  \eqref{solu} is a solution of \eqref{FDE-q-Cap}. At first, note that
$$u(x,t)=A+B\sum\limits_{n=0}^\infty\frac{ c_n (-1)^n }{(1+\alpha)^n}\frac{x^{(n+1)(1+\alpha)-1}}{[(n+1)(1+\alpha)-1]t^{\frac{(n+1)(1+\alpha)-1}{1+\alpha}}},$$
where the series in right side is absolutely convergent over compact sets in $\bbR^+_0 \times (0,T)$. Then, we can interchange $\,^C_0D_{x}^{\al}$ and partial derivatives with the series, obtaining that 

\begin{equation}\label{DCu}
\begin{split}
^C_0D_{x}^{\al} u(x,t)&=B\sum\limits_{n=0}^\infty\frac{ c_n (-1)^n }{(1+\alpha)^{n}}\frac{\Gamma((n+1)(1+\alpha)-1)}{\Gamma((n+1)(1+\alpha)-\alpha)}\frac{x^{n(1+\alpha)}}{t^{\frac{(n+1)(1+\alpha)-1}{1+\alpha}}},
\end{split}
\end{equation}

\begin{equation}
\begin{split}
\frac{\partial}{\partial x}\,^C_0D_{x}^{\al} u(x,t)&=B\sum\limits_{n=1}^\infty\frac{ c_n (-1)^n n}{(1+\alpha)^{n-1}}\frac{\Gamma((n+1)(1+\alpha)-1)}{\Gamma((n+1)(1+\alpha)-\alpha)}\frac{x^{n(1+\alpha)-1}}{t^{\frac{(n+1)(1+\alpha)-1}{1+\alpha}}},
\end{split}
\end{equation}
and
\begin{equation}
\begin{split}
u_t(x,t)&=B\sum\limits_{n=1}^\infty\frac{ c_{n-1} (-1)^{n} }{(1+\alpha)^{n}}\frac{x^{n(1+\alpha)-1}}{t^{\frac{(n+1)(1+\alpha)-1}{1+\alpha}}}.
\end{split}
\end{equation}

Then, if we denote $C_{\alpha,n}:=\frac{c_{n-1}}{(1+\alpha)}-\frac{c_n n\Gamma((n+1)(1+\alpha)-1)}{\Gamma((n+1)(1+\alpha)-\alpha)}$, we have 
\begin{equation}\label{Cn=0}
\begin{split}
C_{\alpha,n}&=\frac{1}{(1+\alpha)}\prod\limits_{j=1}^{n-2}\frac{\Gamma((j+1)(1+\alpha))}{\Gamma((j+1)(1+\alpha)+\alpha)}-\prod\limits_{j=1}^{n-1}\frac{\Gamma((j+1)(1+\alpha))}{\Gamma((j+1)(1+\alpha)+\alpha)}\frac{  n\Gamma((n+1)(1+\alpha)-1)}{\Gamma((n+1)(1+\alpha)-\alpha)}\\
&=\frac{1}{(1+\alpha)}\prod\limits_{j=1}^{n-2}\frac{\Gamma((j+1)(1+\alpha))}{\Gamma((j+1)(1+\alpha)+\alpha)}\left[1-\frac{n(1+\alpha)\Gamma(n(1+\alpha))}{\Gamma(n(1+\alpha)+\alpha)}\frac{\Gamma((n+1)(1+\alpha)-1)}{\Gamma((n+1)(1+\alpha)-\alpha)}\right]\\
&=\frac{1}{(1+\alpha)}\prod\limits_{j=1}^{n-2}\frac{\Gamma((j+1)(1+\alpha))}{\Gamma((j+1)(1+\alpha)+\alpha)}\left[1-\frac{\Gamma(n(1+\alpha)+1)}{\Gamma(n(1+\alpha)+\alpha)}\frac{\Gamma(n(1+\alpha)+\alpha)}{\Gamma(n(1+\alpha)+1)}\right]\\
&=\frac{1}{(1+\alpha)}\prod\limits_{j=1}^{n-2}\frac{\Gamma((j+1)(1+\alpha))}{\Gamma((j+1)(1+\alpha)+\alpha)}\left[1-1\right]=0,\quad \forall n\in\bbN.
\end{split}
\end{equation}

The result \eqref{Cn=0} holds for every  $n\in \bbN$, hence function $u$ is a solution to \eqref{FDE-q-Cap}.

\end{remark}

\begin{prop} If $u$ is the selfsimilar solution given in \eqref{solu}, then we have
\begin{equation}\label{DCu2}
^C_0D_{x}^{\al} u(x,t)=B\left(\frac{\Gamma(\alpha)}{t^{\frac{\alpha}{1+\alpha}}}-\sum\limits_{n=1}^\infty\frac{ c_{n-1} (-1)^{n-1} }{n(1+\alpha)^{n+1}}\frac{x^{n(1+\alpha)}}{t^{n+1-\frac{1}{1+\alpha}}}\right)
\end{equation}
or equivalently,
\begin{equation}\label{eqconddir}
-\,^C_0D_{x}^{\al} u(x,t)=-B\Gamma(\alpha) t^{-\frac{\alpha}{1+\alpha}}+\frac{B t^{-\frac{\alpha}{1+\alpha}}}{1+\alpha}\int_0^{x/t^{\frac{1}{1+\alpha}}} w^\alpha E_{\alpha,1+\frac{1}{\alpha},1}\left(-\frac{w^{1+\alpha}}{1+\alpha}\right)\dd w.
\end{equation}
\end{prop}
\begin{proof}
Using the fact that
$$c_n=c_{n-1}\frac{\Gamma(n(1+\alpha))}{\Gamma(n(1+\alpha)+\alpha)}=c_{n-1}\frac{\Gamma(n(1+\alpha)+1)}{\Gamma(n(1+\alpha)+\alpha)n(1+\alpha)}=\frac{c_{n-1}}{n(1+\alpha)}\frac{\Gamma((n+1)(1+\alpha)-\alpha)}{\Gamma((n+1)(1+\alpha)-1)},$$
and replacing the above expression in \eqref{DCu} we get
\begin{equation}
\begin{split}
-\,^C_0D_{x}^{\al} u(x,t)&=-\frac{B\Gamma(\alpha)}{t^{\frac{\alpha}{1+\alpha}}}-B\sum\limits_{n=1}^\infty\frac{ c_n (-1)^n }{(1+\alpha)^{n}}\frac{\Gamma((n+1)(1+\alpha)-1)}{\Gamma((n+1)(1+\alpha)-\alpha)}\frac{x^{n(1+\alpha)}}{t^{\frac{(n+1)(1+\alpha)-1}{1+\alpha}}}\\
&=-B\Gamma(\alpha)t^{-\frac{\alpha}{1+\alpha}}+B\sum\limits_{n=1}^\infty\frac{ c_{n-1} (-1)^{n-1} }{n(1+\alpha)^{n+1}}\frac{x^{n(1+\alpha)}}{t^{n+1-\frac{1}{1+\alpha}}}.
\end{split}
\end{equation}
Then
\begin{equation}
\begin{split}
-\,^C_0D_{x}^{\al} u(x,t)&=-B\Gamma(\alpha)t^{-\frac{\alpha}{1+\alpha}}+B\sum\limits_{n=1}^\infty\frac{ c_{n-1} (-1)^{n-1} }{n(1+\alpha)^{n+1}}\frac{x^{n(1+\alpha)}}{t^{n+1-\frac{1}{1+\alpha}}}\\
&=-B\Gamma(\alpha)t^{-\frac{\alpha}{1+\alpha}}+\frac{B t^{-\frac{\alpha}{1+\alpha}}}{1+\alpha}\int_0^{x/t^{\frac{1}{1+\alpha}}} \frac{\dd}{\dd w}\left(\sum\limits_{n=1}^\infty\frac{ c_{n-1} (-1)^{n-1} }{n(1+\alpha)^{n}}w^{n(1+\alpha)}\right)\dd w\\
&=-B\Gamma(\alpha)t^{-\frac{\alpha}{1+\alpha}}+\frac{B t^{-\frac{\alpha}{1+\alpha}}}{1+\alpha}\int_0^{x/t^{\frac{1}{1+\alpha}}}  \sum\limits_{n=0}^\infty\frac{ c_{n} (-1)^{n} }{(1+\alpha)^{n}}w^{(n+1)(1+\alpha)-1}\dd w\\
&=-B\Gamma(\alpha)t^{-\frac{\alpha}{1+\alpha}}+\frac{B t^{-\frac{\alpha}{1+\alpha}}}{1+\alpha}\int_0^{x/t^{\frac{1}{1+\alpha}}}  w^\alpha E_{\alpha,1+\frac{1}{\alpha},1}\left(-\frac{w^{1+\alpha}}{1+\alpha}\right)\dd w\\
\end{split}
\end{equation}

\end{proof}

The last aim of this subsection is to prove that the kernel of the selfsimilar solution given in \eqref{defsigma} is non-negative in $\bbR^+$ and the proof will be supported in a weak extremum principle for the space fractional diffusion equation
    
\begin{equation}\label{ecdemo0}
u_t -\frac{\partial}{\partial x}\,^C_0 D_x^\alpha u =f
\end{equation}
in the region
\begin{equation}\label{regionQ}Q_{s,T}=\{(x,t): 0<x<s(t), 0<t<T\},
\end{equation}
for a given function $s:[0,T]\rightarrow \bbR$ such that $s\in C[0,T]$, $s(0)=b>0$ and there exists $ M>0\,  / \, 0<\dot{s}(t)\leq M$ for every $t\in [0,T]$. We define the  parabolic boundary of $Q_{s,T}$  by
$$\partial\gamma_{s,T}=\gamma_1 \cup \gamma_2 \cup \gamma_3,$$
where $\gamma_1=\{(0,t): 0\leq t \leq T\},$ $\gamma_2=\{(x,0): 0\leq x \leq s(0)=b\} $ and  $\gamma_3=\{(s(t),t): 0\leq t \leq T\}.$
%

The next weak extremum principle was stated in \cite{Rys:2020} and we recall it below for the benefit of the reader.
\begin{theo}\label{lemaKaRy}
Let $u$ be a solution to \eqref{ecdemo0} in the region $Q_{s,T}$ defined in \eqref{regionQ}, such that $u$ has the following regularity: $u \in C(\overline{Q_{s,T}}), u_t \in C(Q_{s,T})$ and for every $t\in (0, T)$, for every $0 < \eta < \omega < s(t)$ we have $u(\cdot, t) \in W^{2,\frac{1}{1-\beta}} ( \eta, \omega)$ for some $\beta \in  (\al, 1]$. Let $\partial \gamma_{s,T}$ be its parabolic boundary. Then,
\begin{enumerate}
\item If $f \leq 0$, then $u$ attains its maximum on $\partial \gamma_{s,T}$.
\item If $f \geq 0$, then $u$ attains its minimum on $\partial \gamma_{s,T}$.
\end{enumerate}
\end{theo}

\begin{prop}\label{PropnonegML3}
Let $\al\in (0,1)$ be. Then the function $\sigma_\alpha$ defined in \eqref{defsigma} is a non-negative function in $\bbR^+$.
\end{prop}

\begin{proof}
 We know that $\sigma_\alpha(0^+)=+\infty$ and that $\sigma_\alpha\in C^1(\bbR^+)$. Suppose that there exists $z_0>0$ such that $\sigma_\alpha(z_0)<0$. Then we can affirm that there exists a ``first value'' $c>0$ for which $\sigma_\alpha(c)=0$. Also, from  \cite[Lemma 5.2]{GoKi-Libro} we know that the complex variable Mittag-Leffler function $E_{\al, 1+\frac{1}{\al},1}(z)$ is an entire function, then it has isolated roots and we can choose a sufficiently small $\delta>0$ such that 
 $\sigma_\alpha(z)\geq 0$ for $z\in (0,c]$,  $\sigma_\alpha(z)<0$ for $z\in (c,c+\delta]$ and
\begin{equation}\label{positsigma}
\displaystyle \int_0^{c+\delta}  \sigma_\alpha(w)\dd w> 0.
\end{equation}
Now, let $0<\varepsilon<1$ be, an consider the functions $s^{\varepsilon,\delta}$ and $u^\varepsilon$   defined by  
\begin{equation}\label{s-aux}s^{\varepsilon,\delta} (t)=(c+\delta)(t+\varepsilon)^\frac{1}{1+\alpha},\quad t\geq 0, \quad \text{ for } \, t\in (0,T)\end{equation}
and

\begin{equation}\label{u-aux}u^\varepsilon(x,t)=u(x,t+\varepsilon)=\frac{\displaystyle\int_0^{x/(t+\varepsilon)^{1/(1+\al)}}  \sigma_\alpha(w)\dd w}{\displaystyle\int_0^{c+\delta}  \sigma_\alpha(w)\dd w} \quad \text{ for } \, 0<x<s^{\varepsilon, \delta}(t), 0<t<T,  \end{equation}
where $u$ is the function defined in \eqref{solu} for $A=0$ and $B=\frac{1}{\displaystyle\int_0^{c+\delta}  \sigma_\alpha(w)\dd w}$. Then, if we define the region 
$$Q_{s^{\varepsilon,\delta},T}=\{(x,t): 0<x<s^{\varepsilon,\delta}(t), 0<t<T\},$$
and its parabolic boundary 
$$\partial\gamma_{s^{\varepsilon,\delta},T}=\gamma_1 \cup \gamma_2 \cup \gamma_3,$$
where $\gamma_1=\{(0,t): 0\leq t \leq T\},$, $\gamma_2=\{(x,0): 0\leq x \leq s^{\varepsilon,\delta}(0)=(c+\delta)\varepsilon^\frac{1}{1+\alpha}\}$ and  $\gamma_3=\{(s^{\varepsilon,\delta}(t),t): 0\leq t \leq T\}$ (see Figure \ref{Fig1}), it results that $u^\varepsilon$ is a solution to the moving-boundary problem
\begin{equation}\label{Aux-Prob}
 \begin{array}{lll}
(i) &u_t - \frac{\partial}{\partial x}\,^C_0D_x^\alpha u (x,t)=0, & 0<x<s^{\varepsilon,\delta}(t),0<t<T,\\
(ii) & u(0,t)=0, &  0<t<T,\\
(iii) & u(s^{\varepsilon,\delta}(t),t)=1, &  0<t<T,\\
(iv) & u(x,0)=\frac{\displaystyle\int_0^{x/\varepsilon^{1/(1+\al)}}  \sigma_\alpha(w)\dd w}{\displaystyle\int_0^{c+\delta}\sigma_\alpha(w)\dd w}\geq 0, &   0\leq x\leq s^{\varepsilon, \delta}(0)=(c+\delta)\varepsilon^\frac{1}{1+\alpha}>0.
\end{array}
\end{equation}

\begin{figure}[h!]
\centering
\includegraphics[scale=0.2]{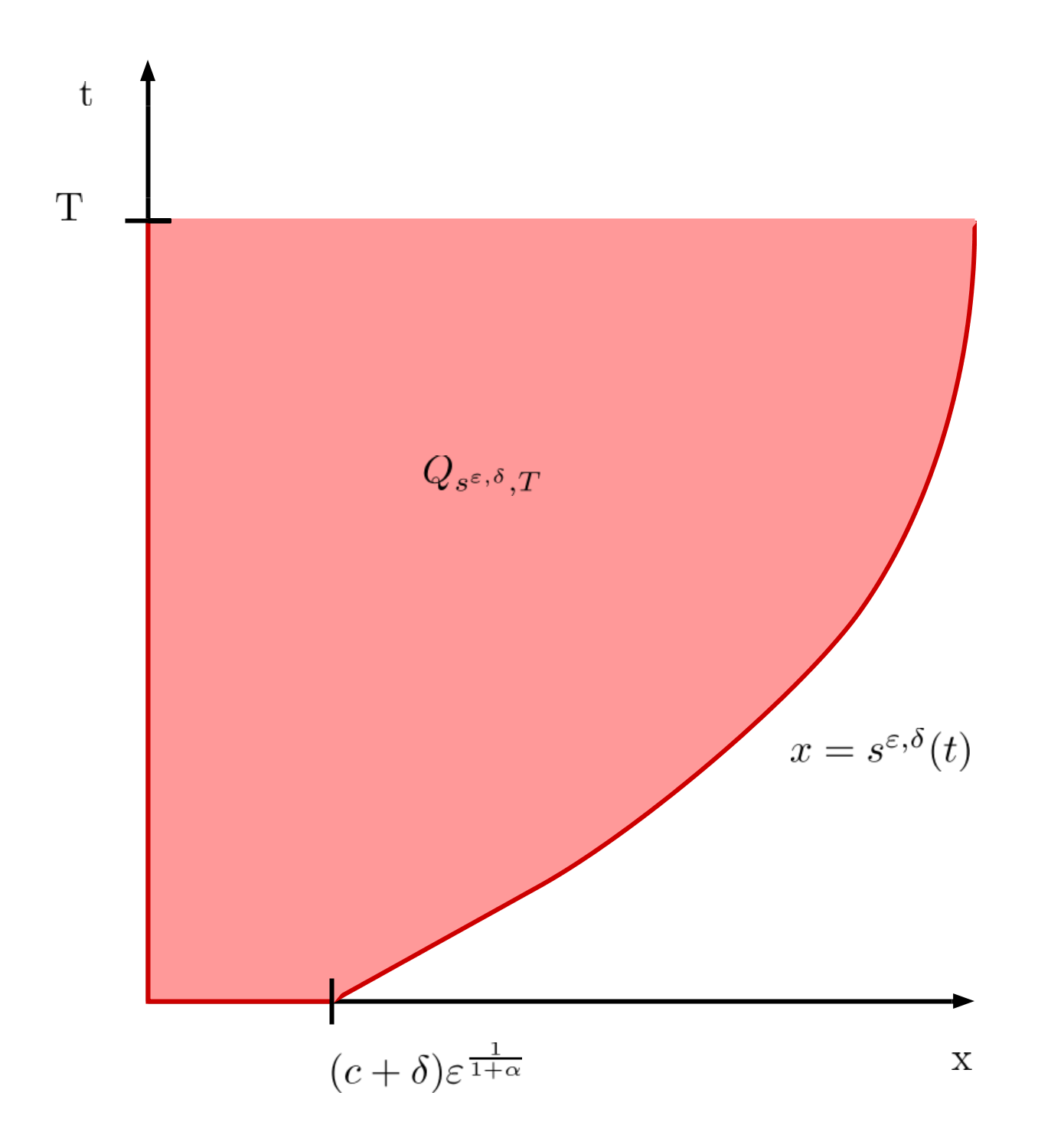}
\caption{Region $Q_{s^{\varepsilon,\delta},T}$ and its parabolic boundary.}
\label{Fig1}
\end{figure}

Clearly, $u^\varepsilon \in C(\overline{Q_{s^{\varepsilon,\delta},T}}), \frac{\partial}{\partial t}u^\varepsilon \in C(Q_{s^{\varepsilon,\delta},T})$, and for some $\beta\in (\alpha,1)$, we have $u^\varepsilon(\cdot,t)\in W^{2,\frac{1}{1-\beta}}(\delta,\omega)$, for $t\in (0,T), 0<\delta<\omega<s^{\varepsilon,\delta}(t)$.

Then, by Theorem \ref{lemaKaRy} (or \cite[Lemma 6]{Rys:2020}), it results that $u^\varepsilon$ attains its mimimum and its maximum at the parabolic boundary  $\partial\gamma_{s^{\varepsilon,\delta},T}$. Hence it easily straightforward that  $u^\varepsilon(x,t)\geq 0$ for all $(x,t)\in Q_{s^{\varepsilon,\delta},T}$. 

Next, we analyze the behavior of $u^\varepsilon$ at the parabolic boundary $\partial\gamma_{s^{\varepsilon,\delta},T}$.
Let $f_t(x):=u^\varepsilon (x,t)$ be. Thus,
$$f'_t(x)=\frac{\sigma_\alpha\left(\frac{x}{(t+\varepsilon)^{1/(1+\al)}}\right)}{(t+\varepsilon)^{1/(1+\al)}\displaystyle\int_0^{c+\delta}  \sigma_\alpha(w)\dd w}$$
\noindent being  $f'_t(x)\geq 0$ in $(0,c(t+\varepsilon)^{1/(1+\al)})$ and $f'_t(x)< 0$ in $(c(t+\varepsilon)^{1/(1+\al)},(c+\delta)(t+\varepsilon)^{1/(1+\al)})$. We conclude then that $f_t$ attains its  maximum over $[0,(c+\delta)(t+\varepsilon)^{1/(1+\al)}]=[0,s^{\varepsilon,\delta}(t)]$ at the point $x=c(t+\varepsilon)^{1/(1+\al)}$. Moreover, $f_t(c(t+\varepsilon)^{1/(1+\al)})>f_t(s^{\varepsilon,\delta}(t))$.


Besides, if we denote by $\xi_t=c(t+\varepsilon)^{1/1+\al}$, for every $t\in [0,T]$, we can state that 
$$f_t(\xi_t)=u^\varepsilon (\xi_t,t)=\frac{\displaystyle\int_0^c  \sigma_\alpha(w)\dd w}{\displaystyle\int_0^{c+\delta}  \sigma_\alpha(w)\dd w}=u^\varepsilon (\xi_{t'},t')=f_{t'}(\xi_{t'}),\quad \forall t,t'\in [0,T],$$
and then, 
$u^\varepsilon$ attains its maximum on $\overline{Q_{s^{\varepsilon,\delta},T}}$ at every point $(\xi_t,t)$, for all $t\in [0,T]$. In particular, $u^\varepsilon$ attains its maximum  at $(\xi_0,0)\in \partial\gamma_{s^{\varepsilon,\delta},T}$. 
Denote by $A:=u^\varepsilon(\xi_t,t)$, for every $t\in [0,T]$. Note that  $A=u^\varepsilon(\xi_t,t)>u^\varepsilon((s^{\varepsilon,\delta}(t),t)=1,\forall t\leq T$.\\

Let us consider now the function
$$v(x,t)=\frac{x^{1+\al}}{1+\al}+\Gamma(1+\al)t$$
and define 

\begin{equation}w^\varepsilon_\kappa(x,t)=u^\varepsilon (x,t)+\kappa v(x,t)\end{equation}
where the constant $\kappa$ will be specified latter. Observe that $v$ and $w^\varepsilon_\kappa$ are both solutions to \eqref{Aux-Prob}$-(i)$, for every $\kappa>0$, and $w^\varepsilon_\kappa$ verifies the hypothesis of Theorem \ref{lemaKaRy}. Then, 
\begin{equation}\label{hip1}
\max_{\overline{Q_{s^{\varepsilon,\delta},T}}}w^\varepsilon_\kappa=\max_{\partial\gamma_{s^{\varepsilon,\delta},T}}w^\varepsilon_\kappa.
\end{equation}

Finally, let us make some computations in order to  evaluate $w^\varepsilon_\kappa$ at the parabolic boundary:
\begin{equation}\label{acotmax1}
\begin{split}
\max_{x\in [0,s^{\varepsilon,\delta}(0)]} w^\varepsilon_\kappa(x,0)&\leq \max_{x\in [0,s^{\varepsilon,\delta}(0)]} u^\varepsilon(x,0) +\kappa\max_{x\in [0,s^{\varepsilon,\delta}(0)]} v(x,0)\\
&=u^\varepsilon (\xi_0,0)+\kappa v(s^{\varepsilon,\delta}(0),0)\\
&=A+\kappa \frac{(c+\delta)^{1+\al}\varepsilon}{1+\al}.
\end{split}
\end{equation}
Also, we have that  
\begin{equation}\label{acotmax2}
\begin{split}
\max_{x\in [0,s^{\varepsilon,\delta}(t)]}w^\varepsilon_\kappa(x,t)&\geq  w^\varepsilon_\kappa(\xi_t,t)\\
&=A+\kappa \left[\frac{c^{1+\al}(t+\varepsilon)}{1+\al}+\Gamma(1+\al)t\right].
\end{split}
\end{equation}

Then, taking $t_0>\varepsilon\frac{(c+\delta)^{1+\al} -c^{1+\al}}{c^{1+\al}+\Gamma(2+\al)}>0$, it holds that 
$$\frac{c^{1+\al}(t_0+\varepsilon)}{1+\al}+\Gamma(1+\al)t_0> \frac{(c+\delta)^{1+\al}\varepsilon}{1+\al},$$
and therefore, taking into account \eqref{acotmax1} and \eqref{acotmax2} we get:
$$\max_{x\in [0,s^{\varepsilon,\delta}(t_0)]}w^\varepsilon_\kappa(x,t_0)>\max_{x\in [0,s^{\varepsilon,\delta}(0)]} w^\varepsilon_\kappa(x,0).$$
and we conclude that $w^\varepsilon_\gamma$ does not attains its maximum at $\gamma_2$.

On the other hand, $v(\cdot,t)$ is a strictly increasing function for every fixed $t$. Then, $w^\varepsilon_\kappa$ does not attains its maximum at $\gamma_1$.

Finally, asking $\kappa$ to verify that  $\kappa<\frac{(A-1)(\al +1)}{[(c+\delta)^{1+\al}-c^{1+\al}](T+\varepsilon)}$, we can affirm that 
\begin{equation}
\begin{split}
w^\varepsilon_\kappa(s^{\varepsilon,\delta}(t),t)&=1+\kappa\left[\frac{(c+\delta)^{1+\al}(t+\varepsilon)}{1+\al}+\Gamma(1+\al)t\right]\\
&<A+\kappa\left[\frac{c^{1+\al}(t+\varepsilon)}{1+\al}+\Gamma(1+\al)t\right]=w^\varepsilon_\kappa(\xi_t,t),\quad \forall t\leq T,
\end{split}
\end{equation}
from where we clame that $w^\varepsilon_\kappa$ does not attains its maximum at $\gamma_3$.

Therefore, $w^\varepsilon_\kappa$ does not attains its maximum at  the parabolic boundary $\partial\gamma_{s^{\varepsilon,\delta},T}$, which contradicts the equality \eqref{hip1}.

This contradiction comes from assuming that there exists $z_0>0$ such that $\sigma_\alpha(z_0)<0$.
Thus,
\begin{equation}\label{noneg3}
\sigma_\alpha(z)\geq 0,\quad \forall z>0,
\end{equation}
and the thesis holds.\end{proof}

\begin{coro} The three parametric Mittag-Leffler function involved in the kernel of the self-similar solution \eqref{solu} verifies that  
 \begin{equation}\label{ML>=0}
 E_{\alpha,1+\frac{1}{\alpha},1}\left(-\frac{x^{1+\alpha}}{1+\alpha}\right)\geq 0 \quad  \text{for all }\,  x>0.
\end{equation} 
\end{coro}

Note that, if $\sigma_\al(x):=x^{\al-1} E_{\alpha,1+\frac{1}{\alpha},1}\left(-\frac{x^{1+\alpha}}{1+\alpha}\right)\geq 0$ for all $x>0$, then \eqref{ML>=0} holds for all $x>0$.

\section{Explicit solution for the fractional space one-phase Stefan problem with a Dirichlet condition at the fixed face}

Let us return to problem \eqref{FSP-D} for a constant Dirichlet boundary data $g\equiv U_0$ and melting temperature $U_m$ such that $U_0>U_m$, given by the following free boundary problem:
\begin{equation}\label{FSP-D-new}
 \begin{array}{lll}
(i) & \frac{\partial}{\partial t}u(x,t)-\frac{\partial}{\partial x}\,^C_0D_x^\alpha u (x,t)=0, & 0<x<s(t),0<t<T,\\
(ii) & u(0,t)=U_0>U_m, &  0<t<T,\\
(iii) & u(s(t),t)=U_m, &  0<t<T,\\
(iv) & s(0)=0, &  \\
(v) & \dot{s}(t)=-(\,^C_0D_x^\alpha u)(s(t),t), &   0<t<T.
\end{array}
\end{equation}

Let $u$ be defined by \eqref{solu}. From \eqref{FSP-D-new}$-(ii)$, we deduce that $A=U_0.$ Now, from condition \eqref{FSP-D-new}$-(iii)$, we have
\begin{equation}\label{free1}
u(s(t),t)=U_0+B \int_0^{s(t)/t^{1/(1+\al)}}\sigma_\al(w)\dd w=U_m.
\end{equation}
Note that \eqref{free1} must be verified for all $t\in (0,T)$, then the free boundary $s$  must be proportional to $t^{1/1+\al}$, that is to say
\begin{equation}\label{freeb}
s(t)=\xi t^{\frac{1}{1+\alpha}},\quad \text{ for some } \xi\in \bbR^+, \quad t\, \in \, (0,T),
\end{equation}
which satisfies \eqref{FSP-D-new}$-(iv)$.
Replacing \eqref{freeb} in \eqref{free1} yields that
\begin{equation}\label{B} B= \frac{-(U_0-U_m)}{\int_0^\xi\sigma_\al(w)\dd w},
\end{equation}
where we have used inequality \eqref{noneg3}, the fact that $\sigma_\alpha$ is positive in a neiborhood of $0$ and that $\xi>0$.

Replacing \eqref{eqconddir} on \eqref{FSP-D-new}(v),  and deriving \eqref{freeb}, we have
\begin{equation}\label{condvD}
\begin{split}
\frac{\xi}{1+\alpha}t^{-\frac{\alpha}{1+\alpha}} &=-B\Gamma(\alpha)t^{-\frac{\alpha}{1+\alpha}}+B\frac{t^{-\frac{\alpha}{1+\alpha}}}{1+\alpha}\int_0^\xi w\sigma_\al(w)\dd w\\
&=\frac{Bt^{-\frac{\alpha}{1+\alpha}}}{1+\alpha}\left(-\Gamma(\alpha)(1+\alpha)+ \int_0^\xi w\sigma_\al(w)\dd w\right)
\end{split}
\end{equation}
Then, combining \eqref{B} and \eqref{condvD}, we have the following condition
\begin{equation}\label{igualdadxi}
\xi= \frac{(U_0-U_m)\left(\Gamma(\alpha)(1+\alpha)- \displaystyle\int_0^\xi w\sigma_\al(w)\dd w\right)}{ \displaystyle\int_0^\xi\sigma_\al(w)\dd w}.
\end{equation}
Therefore, we seek for a positive number $\xi$ which verifies the following equation
\begin{equation}
x= H(x),\quad x>0,
\end{equation}
where the function $H\colon \bbR_0^+\rightarrow \bbR$ is defined by the expression:
\begin{equation}\label{fcH}
H(x)= \frac{(U_0-U_m)\left(\Gamma(\alpha)(1+\alpha)-\displaystyle\int_0^x w\sigma_\al(w)\dd w\right)}{\displaystyle\int_0^x\sigma_\al(w)\dd w}.
\end{equation}
Observe that 
$$\lim\limits_{x\searrow 0} \int_0^x w\sigma_\al(w)\dd w=0^+,\quad \lim\limits_{x\searrow 0}\int_0^x \sigma_\al(w)\dd w=0^+.$$
Then $H(0^+)=\lim\limits_{x\searrow 0}H(x)=+\infty$, because $(U_0-U_m)\Gamma(\alpha)(1+\alpha)>0$. Moreover, it is easy to prove, by using Proposition \ref{PropnonegML3}, that $H$ is a non increasing function in $[0,+\infty)$.
Then, we can affirm that there exists unique $\xi>0$ such that $H(\xi)=\xi$.

From the preceding analysis,  the next theorem follows.

\begin{theo} An explicit solution for the fractional space one-phase Stefan-like problem \eqref{FSP-D-new} is given by 
\begin{equation}\label{solueq53}
u_\al(x,t)=U_0-\frac{(U_0-U_m)}{\displaystyle\int_0^{\xi_\al} \sigma_\al(w)\dd w} \int_0^{x/t^{1/(1+\al)}} \sigma_\al(w)\dd w.
\end{equation}
\begin{equation}\label{freebeq53}
s_\al(t)=\xi_\al t^{\frac{1}{1+\alpha}}, \quad t\, \in \, (0,T),
\end{equation}
where $\xi_\al \in \, \bbR^+$ is the unique solution to the equation
\begin{equation}\label{fcHeq53}
H_\al(x) = x, \quad x > 0,
\end{equation}
and the function $H_\al$ is defined by \eqref{fcH}.

\end{theo}

\begin{remark}

If we take $\al=1$ in \eqref{FSP-D-new}, we recover the classical Lam\'e-Clapeyron-Stefan problem
\begin{equation}\label{FSP-D-new-al1}
\begin{array}{lll}
(i) & \frac{\partial}{\partial t}u(x,t)-\frac{\partial^2}{\partial x^2}u (x,t)=0, & 0<x<s(t),0<t<T,\\
(ii) & u(0,t)=U_0>U_m, &  0<t<T,\\
(iii) & u(s(t),t)=U_m, &  0<t<T,\\
(iv) & s(0)=0, &  \\
(v) & \dot{s}(t)=-\frac{\partial}{\partial x}u(s(t),t), &  0<t<T.
\end{array}
\end{equation}
By Remark \ref{ML3particulares} we know that,
\begin{equation}
\sigma_1(w)=w^0 E_{1,2,1}\left(-\frac{w^2}{2}\right)=e^{-\left(\frac{w}{2}\right)^2}.
\end{equation}
Then, the pair
\begin{equation}\label{solueq53al1}
u_1(x,t)=U_0-\frac{(U_0-U_m)}{\displaystyle\int_0^{\xi_1} \sigma_1(w)\dd w} \int_0^{x/t^{1/2}} \sigma_1(w)\dd w=U_0-\frac{(U_0-U_m)}{erf\left(\frac{\xi_1}{2}\right) }erf\left(\frac{x}{2\sqrt{t}}\right),
\end{equation}
\begin{equation}\label{freebeq53al1}
s_1(t)=\xi_1 t^{\frac{1}{2}}, \quad t\, \in \, (0,T),
\end{equation}
is a solution to \eqref{FSP-D-new-al1} where $\xi_1 \in \bbR^+$ is the unique solution to the equation
\begin{equation}\label{fcHeq53al1}
H_1(x) = x, \quad x > 0,
\end{equation}
with
$$H_1(x)=\frac{(U_0-U_m)\left(2-\displaystyle\int_0^x w e^{-w^2/4}\dd w\right)}{2\frac{\sqrt{\pi}}{2}erf\left(\frac{x}{2}\right)}=\frac{U_0-U_m}{\sqrt{\pi}}\frac{2e^{-\left(\frac{x}{2}\right)^2}}{erf\left(\frac{x}{2}\right)}.$$

That is, we have recovered the classical Lam\'e-Clapeyron-Stefan solution to problem \eqref{FSP-D-new-al1} given in \cite{Lame:1831}. 
\end{remark}

\section{Explicit solution for the fractional space one-phase Stefan problem with a Neumann condition at the fixed face}

Now, we consider the problem \eqref{FSP-N} for a heat flux boundary data given by $h(t)=g_0 t^{-\frac{\al}{1+\al}}$ and melting temperature $g_m$ such that $g_0>g_m$.

\begin{equation}\label{FSP-N-new}
 \begin{array}{lll}
(i) & \frac{\partial}{\partial t}v(x,t)-\frac{\partial}{\partial x}\,^C_0D_x^\alpha v (x,t)=0, & 0<x<s(t),0<t<T,\\
(ii) & \lim\limits_{x\rightarrow 0^+}
\,^C_0D^\al_x v(x,t)=-g_0 t^{-\frac{\al}{1+\al}}, & 0<t<T,\\
(iii) & v(s(t),t)=g_m, &  0<t<T,\\
(iv) & s(0)=0, & \\
(v) & \dot{s}(t)=-(\,^C_0D_x^\alpha v)(s(t),t), &   0<t<T.
\end{array}
\end{equation}

Let $v$ be defined by \eqref{solu}. From \eqref{FSP-N-new}$-(ii)$ and \eqref{DCu2}, we deduce that $B=-\frac{g_0}{\Gamma(\al)}$, because
\begin{equation}
\lim\limits_{x\rightarrow 0^+}\,^C_0D_{x}^{\al} v(x,t)=\lim\limits_{x\rightarrow 0^+} B\left(\frac{\Gamma(\alpha)}{t^{\frac{\alpha}{1+\alpha}}}-\sum\limits_{n=1}^\infty\frac{ c_{n-1} (-1)^{n-1} }{n(1+\alpha)^{n+1}}\frac{x^{n(1+\alpha)}}{t^{n+1-\frac{1}{1+\alpha}}}\right)=\frac{B\Gamma(\alpha)}{t^{\frac{\alpha}{1+\alpha}}}
\end{equation}

From condition \eqref{FSP-N-new}$-(iii)$, we have,
\begin{equation}\label{free1Nnew}
v(s(t),t)=A-\frac{g_0}{\Gamma(\al)}\int_0^{s(t)/t^{1/(1+\al)}}\sigma_\al(w)\dd w=g_m.\end{equation}

from where, we will ask again the free boundary $s$  to be proportional to $t^{1/1+\al}$,
\begin{equation}\label{freeb2}
s(t)=\eta t^{\frac{1}{1+\alpha}},\quad \text{ for some } \eta\in \bbR, \quad t\, \in \, (0,T).
\end{equation}
Then
\begin{equation}\label{A}
A=g_m+\frac{g_0}{\Gamma(\al)}\int_0^\eta \sigma_\al(w)\dd w
\end{equation}

From condition \eqref{FSP-N-new}(v), \eqref{eqconddir} and \eqref{freeb2}, we have

\begin{equation}
\begin{split}
 \frac{\eta}{1+\alpha}t^{-\frac{\alpha}{1+\alpha}} &=g_0 t^{-\frac{\alpha}{1+\alpha}}-\frac{g_0}{\Gamma(\al)}\frac{t^{-\frac{\alpha}{1+\alpha}}}{1+\alpha}\int_0^\eta w\sigma_\al(w)\dd w\\
&=g_0 t^{-\frac{\alpha}{1+\alpha}}\left(1-\frac{1}{(1+\alpha)\G(\al)}\int_0^\eta w\sigma_\al(w)\dd w\right),
\end{split}
\end{equation}
or equivalently
\begin{equation}\label{igualdadxi2}
\begin{split}
\eta =g_0\left((1+\alpha)-\frac{1}{\Gamma(\alpha)}\int_0^\eta w\sigma_\al(w)\dd w\right),
\end{split}
\end{equation}
Therefore, $\eta$ must verify the following equation
\begin{equation}
\begin{split}
x =G(x), \,\, x>0,
\end{split}
\end{equation}
where the function $G$ is defined in $\bbR^+_0$ by the expression:
\begin{equation}\label{fcH2}
G(x)=g_0\left((1+\alpha)-\frac{1}{\Gamma(\alpha)}\int_0^x w\sigma_\al(w)\dd w\right).
\end{equation}
Observe that $G$ is continuous in $[0,+\infty)$. From Proposition \ref{PropnonegML3}, it easily follows that $G$ is an decreasing function. Moreover
$$G(0)=g_0(1+\alpha)>0.$$

From the preceding analysis, we conclude that there exists a unique $\eta\in\bbR^+$ such that $\eta=G(\eta)$, and the next theorem follows.

\begin{theo} An explicit solution for the space-fractional Stefan-like problem \eqref{FSP-N-new} is given by 
\begin{equation}\label{solFSP-N-new}
\begin{split}
v_\al(x,t)&=g_m+\frac{g_0}{\Gamma(\alpha)}\int_0^{\eta_\al} w^{\alpha-1}E_{\alpha,1+\frac{1}{\alpha},1}\left(-\frac{w^{1+\alpha}}{1+\alpha}\right)\dd w-\frac{g_0}{\Gamma(\alpha)}\int_0^{x/t^{1/(1+\al)}} w^{\alpha-1}E_{\alpha,1+\frac{1}{\alpha},1}\left(-\frac{w^{1+\alpha}}{1+\alpha}\right)\dd w\\
&=g_m+\frac{g_0}{\Gamma(\alpha)} \int_{x/t^{1/(1+\al)}}^{\eta_\al} w^{\alpha-1}E_{\alpha,1+\frac{1}{\alpha},1}\left(-\frac{w^{1+\alpha}}{1+\alpha}\right)\dd w\\
\end{split}
\end{equation}
\begin{equation}\label{freebP2}
s_\al(t)=\eta_\al t^{\frac{1}{1+\alpha}},\quad t\, \in \, (0,T).
\end{equation}
where $\eta_\al\in \bbR^+$ is the unique solution to the equation
$$G_\al(x) = x, \quad x > 0,$$
and the function $G_\al$ is defined by \eqref{fcH2}

\end{theo}

\begin{remark}

For $\al=1$ in \eqref{FSP-N-new}, we have

\begin{equation}\label{FSP-N-new-al1}
 \begin{array}{lll}
(i) & \frac{\partial}{\partial t}v(x,t)-\frac{\partial^2}{\partial x^2}v (x,t)=0, & 0<x<s(t),0<t<T,\\
(ii) & \lim\limits_{x\rightarrow 0^+}\frac{\partial}{\partial x}v(x,t)=-g_0t^{-\frac{1}{2}}, &  0<t<T,\\
(iii) & v(s(t),t)=g_m, &  0<t<T,\\
(iv) & s(0)=0, &  \\
(v) & \dot{s}(t)=-\frac{\partial}{\partial x}v(s(t),t), &   0<t<T,
\end{array}
\end{equation}
and
\begin{equation}
\sigma_1(w)=w^0 E_{1,2,1}\left(-\frac{w^2}{2}\right)=e^{-\left(\frac{w}{2}\right)^2}.
\end{equation}
Then, the pair
\begin{equation}\label{2solueq53al1}
v_1(x,t)=g_m+g_0 \int_{x/t^\frac{1}{2}}^{\eta_1} w^0 E_{1,2,1}\left(-\frac{w^2}{2}\right)\dd w=g_m+g_0\sqrt{\pi} \left[erf\left(\frac{\eta_1}{2}\right)-erf\left(\frac{x}{2\sqrt{t}}\right)\right],
\end{equation}
\begin{equation}\label{2freebeq53al1}
s_1(t)=\eta_1 t^{\frac{1}{2}},\quad t\, \in \, (0,T)
\end{equation}
is a solution to the classical Lam\'e-Clapeyron-Stefan problem \eqref{FSP-N-new-al1}, where $\eta_1 \in \bbR^+$ is the unique solution to the equation
\begin{equation}\label{2fcHeq53al1}
G_1(x) = x, \quad x > 0,
\end{equation}
with
$$G_1(x)=g_0\left(2-\int_0^x w\sigma_1(w)\dd w\right)=2g_0 e^{-\left(\frac{x}{2}\right)^2}.$$

as can be stated in \cite{Tar:1981,Tarzia}.
\end{remark}

\begin{remark}
For the solution \eqref{solFSP-N-new} to the space-fractional Stefan-like problem \eqref{FSP-N-new}, we cannot change the condition \eqref{FSP-N-new}(ii) by a Neuman condition of the form $v_x(0+,t)=g(t)$. In fact, observe that
\begin{equation}\label{derivux}
v_x(x,t)=-g_0\frac{x^{\alpha-1}}{t^{\frac{\alpha}{1+\alpha}}}\sum\limits_{k=0}^\infty c_k\left(\frac{ -x^{1+\alpha} }{(1+\alpha)t}\right)^k,
\end{equation}
and $c_k <2$ for all $k$. Then, the series in the right hand of \eqref{derivux} is convergent for $x<1$. Moreover, for $x=0$, the series is equal to $1$. Hence, since $\alpha -1<0$, we conclude that
$$\lim\limits_{x\rightarrow 0^+}v_x(x,t) =-\infty.$$

\end{remark}

\section{Conclusions}

We obtained exact self-similarity solutions for a one-phase one-dimensional fractional space Stefan problem in terms of the three parametric Mittag-Leffler function $E_{\al,m,l}(z)$. We considered Dirichlet and Newmann boundary conditions at the fixed face, involving  Caputo fractional space derivatives of order $0<\al<1$. In both cases, the free boundary term is proportional to $t^\frac{1}{1+\al}$.
Finally, we recover the solution for the classical  one-phase Stefan problem  when the order of the Caputo derivatives approaches one.

\section{Acknowledgements}

\noindent The present work has been sponsored by the Projects PIP N$^\circ$ 0275 from CONICET--Universidad Austral, ANPCyT PICTO Austral 2016 N$^\circ 0090$, Austral N$^\circ 006-$INV$00020$ (Rosario, Argentina) and European Unions Horizon 2020 research and innovation programme under the Marie Sklodowska-Curie Grant Agreement N$^\circ$ 823731 CONMECH.

\bibliographystyle{plain}

\bibliography{Roscani_BIBLIO_GENERAL_nombres_largos_2020}

\end{document}